\documentclass[10pt,a4paper]{article}
\usepackage[utf8]{inputenc}
\usepackage{amsthm,amsmath,enumerate,epic}
\usepackage[mathscr]{euscript}
\usepackage{pstricks,pst-node}
\usepackage{amssymb}
\usepackage{graphicx}
\usepackage{caption}
\usepackage{subcaption}
\usepackage{hyperref}
\usepackage[vcentermath]{youngtab}

\newtheorem{thm}{\bf Theorem}[subsection]

\newtheorem{remark}{\bf Remark}[thm]
\newtheorem{defi}{\bf Definition}[subsection]

\newtheorem{eg}{\bf Example}[subsection]
\newtheorem{theorem}{Theorem}[subsection]
\newtheorem{lemma}{\bf Lemma}[subsection]

\newenvironment{Proof}{%
\par\noindent{\scshape Proof:}\begin{rm}}{\hfill$\Box$\end{rm}\newline}
\numberwithin{equation}{subsection}
\title{ Multiplicity at any torus-fixed point in a Richardson variety in the Symplectic Grassmannian}
\date{}
\author {Papi Ray\\
and\\
Shyamashree Upadhyay\\
\\
Department of Mathematics, Indian Institute of Technology\\
Guwahati, Guwahati-781039, Assam, India\\
Emails: popiroy93@iitg.ac.in, shyamashree@iitg.ac.in }
\begin{document}
\maketitle
\begin{abstract}
We give a combinatorial description of the multiplicity at any torus fixed point on a Richardson variety in the Symplectic Grassmannian.
\end{abstract}


\tableofcontents
\section{Introduction}\label{s.Introduction}
This paper is a sequel to \cite{ru}.In \cite{ru}, an explicit gr\"obner basis for the ideal of the tangent cone at any torus fixed point of a Richardson variety in the Symplectic Grassmannian has been given. In this paper, we use the main theorem of \cite{ru} to give a combinatorial description of the multiplicity at any torus fixed point on a Richardson variety in the Symplectic Grassmannian.\\
In \cite{gr}, Ghorpade and Raghavan compute the multiplicity at any point on a Schubert variety in the Grassmannian. This paper generalizes their result. In \cite{kv}, Kreiman has given an explicit gr\"obner basis for the ideal of the tangent cone at any torus fixed point of a Richardson variety in the ordinary grassmannian, thus generalizing a result of Kodiyalam-Raghavan (\cite{kr}). In \cite{kv}, Kreiman has also used the gr\"obner basis result to deduce a formula which computes the multiplicity of a Richardson variety at any torus fixed point in the ordinary grassmannian by counting families of certain non-intersecting lattice paths.\\
In this paper, we apply techniques similar to \cite{kv} and deduce a formula (analogous to \cite{kv}) which computes the multiplicity of a Richardson variety at any torus fixed point in the Symplectic Grassmannian by counting families of certain non-intersecting lattice paths. The main theorems of our paper are theorem \ref{t.main3} and theorem \ref{t.main4}. In theorem \ref{t.main3}, we give a combinatorial description of the multiplicity as the cardinality of certain special kind of sets. And, in theorem \ref{t.main4}, we provide the result on counting the multiplicity as the cardinality of a family of certain non-intersecting lattice paths.
\section{The Symplectic Grassmannian and Richardson varieties in it}\label{sim.Grassmannian}
Let $d$ be a fixed positive integer. This integer $d$ will be kept fixed throughout this paper. For any $j\in\{1,\ldots,2d\}$, set $j^*=2d+1-j$. Fix a vector space $V$ of dimension $2d$ over an algebraically closed field $\mathfrak{F}$ of arbitrary characteristic. The \textit{Symplectic Grassmannian} $\mathfrak{M}_d(V)$ is defined in \S 2 of \cite{gr}. Let $Sp(V), B, T, I(d,2d), I(d)$ be as defined in \S 2 of \cite{gr}. Let $B^-$ denote the borel subgroup of $Sp(V)$ opposite to $B$.\\

The $T$-fixed points of $\mathfrak{M}_d(V)$ are parametrized by $I(d)$ (as explained in \S 2 of \cite{gr}). The $B$-orbits (as well as $B^-$-orbits) of $\mathfrak{M}_d(V)$ are naturally indexed by its $T$ -fixed points: each
$B$-orbit (as well as $B^-$-orbit) contains one and only one such point. Let $ \alpha \in I(d)$ be arbitrary and let $e_{\alpha}$ denote the corresponding $T$-fixed point of $\mathfrak{M}_d(V)$. The Zariski closure of the $B$ (resp. $B^-$) orbit through $e_\alpha$, with canonical reduced
scheme structure, is called a \emph{Schubert variety} (resp.
\emph{opposite Schubert variety}), and denoted by $X^\alpha$ (resp.
$X_\alpha$).  For $\alpha,\gamma \in I(d)$, the scheme-theoretic
intersection $X_\alpha^\gamma=X_\alpha\cap X^\gamma$ is called a
Richardson variety. It can be seen easily that the set consisting of all pairs of elements of $I(d)$ becomes an indexing set for Richardson varieties in $\mathfrak{M}_d(V)$. It can also be shown that $X_\alpha^\gamma$ is
nonempty if and only if $\alpha\leq \gamma$; and that for $\beta \in I(d)$, $e_\beta\in X_\alpha^\gamma$
if and only if $\alpha \leq \beta\leq \gamma$.
\subsection{Some notation}
For this subsection, let us fix an arbitrary element $v$ of $I(d,2d)$. We will be dealing extensively with ordered pairs $(r,c)$,
$1\leq r,c\leq 2d$,  such that $r$ is not and $c$ is an entry of~$v$.
Let $\mathfrak{R}(v)$ denote the set of all such ordered pairs, that is, $\mathfrak{R}(v)=\{(r,c)|r \in\{1,\ldots,2d\}\setminus v, c\in v\}$. Set $\mathfrak{N}(v):= \{(r,c)\in \mathfrak{R}(v) | r>c\}$, $\mathfrak{OR}(v):=\{(r,c)\in \mathfrak{R}(v) | r\leq c^*\}$, $\mathfrak{ON}(v):=\{(r,c)\in \mathfrak{R}(v) | r>c, r\leq c^*\}=\mathfrak{OR}(v)\cap \mathfrak{N}(v)$, $\mathfrak{d}(v):=\{(r,c)\in\mathfrak{R}(v) | r=c^*\}$. We will refer to $\mathfrak{d}(v)$ as the  {\em diagonal}.

We will be considering {\em monomials} in some of these sets. The definitions of a {\em monomial in a set}, the {\em degree of a monomial}, the {\em intersection} of two monomials in a set, a monomial being {\em contained in} another, and {\em monomial minus} of one monomial from another are as given in \S 3.2 of \cite{su}.\\

\subsection{Recap}\label{s.recap}

In this section, we first recall the definitions of twisted chain, chain boundedness and some partial orders on the negative elements of $\mathbb{N}^2$ from \cite{kv}.\\
Let $(e,f), \ (g,h) \in \mathbb{N}^2$, both negative, then $(e,f) \prec (g,h)$ if $f<h$ and $e>g$,$(e,f) \unlhd (g,h)$ if $f \leq h$ and $e \geq g$. Let $(c,d), \ (e,f) \in (\mathbb{N}^2)^-$, then define $(c,d)\wedge (e,f)$ = (max$(c,e)$,min $(d,f)) \in \mathbb{N}^2$.\\
If $T=\{(e_1,e_2),\hdots (e_m,e_{m+1})\}$ is a subset of $\mathbb{N}^2$, then  $T$ is said to be \textbf{completely disjointed} if $e_i \neq e_j$ when $i \neq j$. A \textbf{negative twisted chain} is a completely disjointed negative subset of $\mathbb{N}^2$ such that for any $u,v \in T$, either $u \prec v$, or $v \prec u$, or $u \wedge v \notin (\mathbb{N}^2)^-$. If $T$ is a positive subset of $\mathbb{N}^2$, then  $T$ is called a \textbf{positive twisted chain} if $\iota(T)$ is a negative twisted chain. A \textbf{twisted chain} is a subset of $\mathbb{N}^2$ which is either a positive or a negative twisted chain.\\
Let $T=\{(e_1,f_1), \hdots ,(e_m,f_m)\}$ be a completely disjointed negative subset of $\mathbb{N}^2$ such that $f_1< \hdots , f_m$. For $\sigma \in S_m$, the permutation group on $m$ elements, define $\sigma (T)=\{(e_{\sigma(1)},f_1), \hdots ,(e_{\sigma(m)},f_m)\}$. Let $\tau =\{\sigma (T) |\  \sigma \in S_m, \sigma (T) \ negative \}$. Impose the following total order on $\tau :$ If $R=\{(a_1,f_1), \hdots ,(a_m,f_m)\}$, $S=\{(b_1,f_1), \hdots ,(b_m,f_m)\} \in \tau$, then $R \stackrel{\text{lex}}{<} S$ if, for the smallest $i$ for which $a_i \neq b_i , a_i > b_i$. Since $\stackrel{\text{lex}}{<} S$ is total order, $\tau$ has a unique minimal element, which  Kreiman has denoted by $\widetilde{T}$ in \cite{kv}. From lemma $9.2$ of \cite{kv} we know that, $\widetilde{T}$ is a negative twisted chain. \\
In \cite{kv}, for $R$ a negative subset of $\mathbb{N}^2$ and $x \in (\mathbb{N}^2)^-$, Kreiman has defined \textbf{depth}{$_R (x)$}, which is maximum $r$ such that there exists a chain $u_1 \prec \hdots \prec u_r$ in $R$ with $u_r \unlhd x$ and, for any two negative subsets $R$ and $S$ of $\mathbb{N}^2$, $R \unlhd S$ (or $S \unrhd R$), if depth$_R(x) \geq$ depth$_S(x)$ for every negative $x \in \mathbb{N}^2 $, which is equivalent to depth$_R(x) \geq$ depth$_S(x)$ for every $x \in S$. If $R,S$ are positive subsets of $\mathbb{N}^2$, then $S \unrhd R$ if $\iota (S) \unlhd \iota (R)$. If $R$ is a negative subset of $\mathbb{N}^2$ and $S$ is a positive subset, then $R \unlhd S$. Also from lemma $9.4$ of \cite{kv} we know that, if $R$ and $S$ are twisted chains then $R \unlhd S \Longleftrightarrow R \leq S$, where the relation $\leq$ on multisets on $\mathbb{N}^2$ is defined in \S $4$ of \cite{kv}.\\
Let $R$ and $S$ be negative and positive twisted chains respectively. Then a multiset $U$ on $\mathbb{N}^2$ is said to be \textbf{chain-bounded by R,S} if $R \unlhd U^-$ and $U^+ \unlhd S$, or equivalently, if for every chain $C$ in $U$, $R \unlhd C^-$ and $C^+ \unlhd S$.\\
For the rest of this papers let $\beta$ be an arbitrary element of $I(d)$ and let $\bar{\beta}=\{1,\ldots,2d\}\setminus\beta$.
\subsection{Some necessary definitions and lemmas}
\begin{defi}
A multiset $U$ in $\bar{\beta} \times \beta$ is called a \textbf{star set} if\\
(i) $(r,c) \in U$ and $r \neq c^\star \Rightarrow$ $(c^\star,r^\star) \in U$ and\\
(ii) Multiplicity of every $(r,c) \in U$ is one.\\
A star set $U$ in $(\bar{\beta} \times \beta)^-$ is called a \textbf{negative star set}. Similarly, a  star set $U$ in $(\bar{\beta} \times \beta)^+$ is called a \textbf{positive star set}.
\end{defi}  
 \begin{defi}
A multiset $U$ in $\bar{\beta} \times \beta$ is called \textbf{$\star \star$-multiset} if \\
(i) $(r,c) \in U \Rightarrow$ $(c^\star,r^\star) \in U$ and\\
(ii) Multiplicity of every $(r,c) \in U$ is one except the diagonal elements and for the diagonal elements, multiplicity of every element is two.
\end{defi} 
\noindent The definition below is as given in definition $4.4$ of \cite{gr}:-
\begin{defi}
A multiset $U$ in $\bar{\beta}\times \beta$ is called a \textbf{special multiset} if\\
$(1)\ U=U^{\#}$ and \\
$(2)$ the multiplicity of any diagonal element in $U$ is even.
\end{defi}
\noindent The definition below is as given in \S $4.1$ of \cite{kr}:-
\begin{defi}\label{d.distinguished}
We call \textbf{distinguished} the subsets $\mathfrak{S}$ of ${\mathfrak{N}(v)}$ satisfying the following conditions:\\
\noindent (A) For $(r,c)\neq (r',c')$ in $\mathfrak{S}$, we have $r\neq r'$ and $c\neq c'$.\\
\noindent (B) If $\mathfrak{S}=\{(r_1,c_1),\ldots,(r_p,c_p)\}$ with $r_1<r_2<\ldots<r_p$, then for $j$, $1\leq j\leq p-1$, we have either $c_j>c_{j+1}$ or $r_j<c_{j+1}$.
\end{defi}
\begin{lemma}\label{lemma.1}
Let $\alpha, \beta, \gamma \in I(d)$ with $\alpha \leq \beta \leq \gamma$, then $\widetilde{T}_\alpha$ is a negative star set in $(\bar{\beta}\times \beta)$ and $\widetilde{W}_\gamma$ is a positive star set in $(\bar{\beta}\times \beta)$.
\end{lemma}
\begin{proof}
It is enough to show that $\widetilde{W}_\gamma$ is a star set in $\bar{\beta} \times \beta$ (for $\widetilde{T}_\alpha$, the proof is similar). Now our claim is, that, $\widetilde{W}_\gamma$ is the distinguished monomial corresponding to $\gamma$ in the sense of proposition $4.3$ of \cite{kr}.\\
Since $\gamma \in I(d)$, so $\gamma =\gamma^{\#}$. Hence if we assume the claim then the proof of the lemma follows from proposition 5.7 of \cite{gr}.\\
\textbf{Proof of the claim:} Since $\widetilde{W}_\gamma$ is a positive twisted chain, therefore $\iota(\widetilde{W}_\gamma)$ is a negative twisted chain.\\
Say, $\iota(\widetilde{W}_\gamma)=\{(c_1,r_1),\hdots ,(c_m,r_m)\}$ with $r_1 < \hdots < r_m$. Hence $\iota(\widetilde{W}_\gamma)$ is a completely disjointed negative subset of $\bar{\beta}\times \beta$ such that for any $(c_i,r_i),(c_j,r_j)\in \iota(\widetilde{W}_\gamma)$ with $i \neq j$, we have :\\
Either $(a) \ (c_i,r_i)\prec (c_j,r_j)$, that is, $r_i<r_j$ and $c_i> c_j$\\
or $(b) \ (c_j,r_j)\prec (c_i,r_i)$, that is, $r_j<r_i$ and $c_j>c_i$\\
or $(c) \ (c_i,r_i) \wedge(c_j,r_j)\notin (\mathbb{N}^2)^- $.\\
Without loss of generality, let us assume $r_i< r_j$ (proof will be similar for $r_i>r_j$). Then either $(a)$ or $(c)$ holds. That is, either by $(a),\ c_i>c_j$ or by $(c),\ (c_j, r_i) \notin (\mathbb{N}^2)^-$. Now $(c_j, r_i) \notin (\mathbb{N}^2)^-$ means $c_j \nless r_i$, that is $c_j \geq r_i$. But here $\iota(\widetilde{W}_\gamma)$ is completely disjointed, so equality cannot happen. So $c_j>r_i$. So, we have if $r_i< r_j$, then either $c_i>c_j$ or $r_i<c_j$. This is nothing but the condition for distinguished monomials.     
\end{proof} 
\begin{lemma}\label{lemma.2}
There exists a bijection from the set of all $\star \star$-multisets in $\bar{\beta} \times \beta$ to the set of all star sets in $\bar{\beta} \times \beta$. 
\end{lemma}
\begin{proof}
Consider the map $\phi$ from the set of all $\star \star$-multisets in $\bar{\beta} \times \beta$ to the set of all star sets in $\bar{\beta} \times \beta$ given by $\phi(U) =$ the underlying set of $U$. Clearly, this map $\phi$ is a bijection.
\end{proof}
\section{The main theorem}\label{ss.tmain}
In this section, we prove one of the main theorems of this paper, namely theorem \ref{t.main3}.
\begin{theorem}\label{t.main1}
\noindent Degree $Y_\alpha ^\gamma (\beta)$ (=$mult_{e_\beta}X_\alpha ^\gamma$) is the number of square free monomials of maximal degree in $P \setminus in_\rhd G_{\alpha ,\beta}^\gamma$(good), where $P:=\ \mathfrak{k}[X_{(r,c)} | (r,c)\in \mathfrak{OR}(\beta)]$
\end{theorem}
\begin{Proof}
Since the initial term of $f_{\mathfrak{w},\beta}$ (where $\mathfrak{w}$ is good) is a positive or negative upper extended $\beta$ chain, therefore it is square free. Then by lemma 8.5 of \cite{kv} we have the proof.
\end{Proof}
\begin{theorem}\label{t.main2}
There exists a degree doubling bijection from the set of all monomials in $P \setminus in_\rhd G_{\alpha,\beta}^\gamma$(good) to the set of all non-vanishing special multisets on $\bar{\beta} \times \beta$ bounded by $T_\alpha, \ W_\gamma$.
\end{theorem}
\begin{proof}
Let $a=$ cardinality of all degree $m$ monomials on $P \setminus in_\rhd G_{\alpha,\beta}^\gamma$(good).\\
$b=$ cardinality of all degree $m$ monomials on $P \setminus in_\rhd I$.\\
$c=$ cardinality of all degree $m$ standard monomials on $Y_\beta ^\gamma (\beta)$.\\
$d=$ cardinality of all degree $2m$ non-vanishing semistandard notched bitableaux on $(\bar{\beta} \times \beta)^\star$.\\
$e=$ cardinality of all degree $2m$ non-vanishing special multisets on $\bar{\beta}\times \beta$ bounded by $T_\alpha, \ W_\gamma$ .\\
So we have to prove $a=e$. Now $G_{\alpha,\beta}^\gamma$(good) $\subseteq I$ (from \cite{ru}), which implies $ in_\rhd G_{\alpha,\beta}^\gamma (good) \ \subseteq in_\rhd I$. So
\begin{equation}\label{e.equ1}
P \setminus in_\rhd G_{\alpha,\beta}^\gamma(good)\ \supseteq P \setminus in_\rhd I
\end{equation}
Again both the monomials of $P \setminus in_\rhd I$ and the standard monomials of $Y_\alpha ^\gamma (\beta)$ form a basis for $P / I$, and thus agree in cardinality in any degree. Therefore $b=c$.
Again from theorem $5.0.6$ of \cite{ru}, we have $d=e$. Also from theorem $5.0.5$ of \cite{ru}, we have $d \leq c$. Now we want to show $d \geq c$. Using equation \ref{e.equ1} we have, $a\geq b$. Again by theorem $5.0.1$ of \cite{ru} we have, $a\leq e$. Therefore $d=e\geq a \geq b =c$, hence $d\geq c$. Therefore $d = c$ and hence, $a\geq b=c=d=e$. We also have $a \leq e$. Hence $a=e$.
\end{proof}
\begin{remark}\label{rem.1}
It follows from the theorem \ref{t.main2} above that, the number of square free monomials of maximal degree in $P \setminus in_\rhd G_{\alpha ,\beta}^\gamma$(good) equals the number of $\star \star$-multisets in $\bar{\beta} \times \beta$ which are of maximal degree among those bounded by $T_{\alpha} , W_\gamma$. [Because, the bijection of theorem \ref{t.main2} above is given by the map $U\mapsto U\cup U^{\#}$.]
\end{remark}
\begin{theorem}\label{t.main3}
$Mult_{e_\beta}X_\alpha ^\gamma$ is the number of star sets $U$ in $\bar{\beta} \times \beta$, which are of maximal degree among those, which are chain-bounded by $\widetilde{T}_\alpha$ and $\widetilde{W}_\gamma$.
\end{theorem}
\begin{proof}
Recall from \cite{kv} that, if $U$ is a multiset on $\bar{\beta} \times \beta$, then the monomial $X_U$ is square-free if and only if $U$ is a subset of $\bar{\beta} \times \beta$, that is each of its elements has degree one. By theorem \ref{t.main1}, $Mult_{e_\beta}X_\alpha ^\gamma$ is the number of square free monomials of maximal degree in $P \setminus in_\rhd G_{\alpha,\beta}^\gamma (good)$. By remark \ref{rem.1}, this equals the number of $\star \star$ multisets in $\bar{\beta} \times \beta$, which are of maximal degree among those bounded by $T_\alpha, \ W_\gamma$. Again by lemma \ref{lemma.2}, this equals the number of star sets in $\bar{\beta} \times \beta$, which are of maximal degree among those bounded by $T_\alpha, \ W_\gamma$. However, a subset of $\bar{\beta} \times \beta$ is bounded by $T_\alpha, \ W_\gamma$ if and only if it is bounded by $\widetilde{T}_\alpha$ and $\widetilde{W}_\gamma$ if and only if it is chain bounded by $\widetilde{T}_\alpha$ and $\widetilde{W}_\gamma$, where the last equivalence is due to Lemma $9.4$ of \cite{kv}.
\end{proof}
\section{Path families and multiplicities}
For this section, we let $R$ and $S$ be fixed positive and negative twisted chains contained in $\bar{\beta} \times \beta$ respectively. Let\\
$\mathcal{M}_R= \ max \{U \subset (\bar{\beta} \times \beta)^-\ |\ R\unlhd U $ and $U$ is a star set$\}$,\\
$\mathcal{M}^S=max\{V\subset (\bar{\beta} \times \beta)^+ \ | \  V\unlhd S$ and $V$ is a star set$\}$, \\
and $\mathcal{M}_R ^S= max \{W \subset (\bar{\beta} \times \beta)\ |\ R \unlhd W^- \ and \ W^+ \unlhd S$ and $W$ is a star set$\}$,\\
where in each case by `max' we mean the star sets $U, V$, or $W$ respectively of maximal degree. For example, $\mathcal{M}_R ^S$ consists of the collection of all star sets $W$ of $(\bar{\beta}\times \beta)$ which are of maximal degree among those which are chain bounded by $R,S$. When $R=\widetilde{T}_\alpha$ and $S=\widetilde{W}_\gamma$, $\mathcal{M}_R ^S$ consists precisely of the star sets $U$ of theorem \ref{t.main3}. In order to give a better formulation of \ref{t.main3}, we study the combinatorics of $\mathcal{M}_R ^S$ . Clearly,\\
$\mathcal{M}_R ^S= \{U \dot{\cup} V \ |\ U\in \mathcal{M}_R,\ V \in \mathcal{M}^S\}$.\\
To study $\mathcal{M}_R^S$, just like in \cite{kv}, we begin by considering $\mathcal{M}_R$, and thus restricting attention to negative star sets of $(\bar{\beta}\times \beta) $. Just like in \cite{kv}, a subset $P\ \subset (\bar{\beta}\times \beta)^-$ is \textbf{depth-one} if it contains no two-element chains and if $P$ is depth-one, then it is a \textbf{negative-path} if the consecutive points are `as close as possible' to each other, so that the points form a continuous path on $(\bar{\beta}\times \beta)^-$ which moves only down or to the right. Similarly if $P \in (\bar{\beta} \times \beta)^+$ is depth-one, then it is a \textbf{positive-path} if the consecutive points are `as close as possible' to each other, so that the points form a continuous path on $(\bar{\beta}\times \beta)^+$ which moves only up or to the left.\\
For any $r=(e,f) \in (\bar{\beta}\times \beta)^-$, all of  $\lfloor r \rfloor$ and $\lceil r \rceil$ are defined as in \cite{kv}. But if $r=(e,f) \in (\bar{\beta}\times \beta)^+$, then we define $\lfloor r \rfloor =(e,f^\prime)$, where $f^\prime =$  max$\{y \in \beta \ | \ (e,y)\in (\bar{\beta} \times \beta)^+ \}$  and $\lceil r \rceil = (e^\prime,f)$, where $e^\prime =$  min $\{x \in \bar{\beta} \ | \ (x,f)\in (\bar{\beta} \times \beta)^+ \}$. \textbf{Now we form the path $P_r$} as follows -\\
$(1)$which begins at $\lfloor r \rfloor$, and ends at $\lceil r \rceil$ and \\
$(2)$ if $r=(e_1,f_1), \ r^\prime =(e_1^\prime,f_1^\prime)$ and $r=(r^\prime)^{\#}$, then $P_r ={P_{r^\prime}} ^{\#}$. Also if $r=(e,f)$ and $e=f^\star$, then $P_r$ is a star set in $(\bar{\beta} \times \beta)$.\\
After doing all of this, if we do the similar things as in \cite{kv}, the only difference being that in every case, $U,V,W\subseteq (\bar{\beta}\times \beta)$ are star sets, then we get the theorem below (theorem \ref{t.main4}), which is the main theorem about counting the  multiplicity as the cardinality of a family of certain non-intersecting lattice paths.
\begin{theorem}\label{t.main4}
$Mult_{e_\beta}X_\alpha ^\gamma$ is the number of disjoint unions $\dot{\bigcup}_{r \in \widetilde{T_\alpha}\cup \widetilde{W}_\gamma}P_r$, where $P_r$ is either negative-path or a positive-path from $\lfloor r \rfloor$ to $ \lceil r \rceil$, depending on whether $r$ is negative or positive.
\end{theorem}
\begin{eg}
Let $d=5$, that is, $2d=10$. Let $\alpha=(1,2,4,6,8), \beta=(2,4,5,8,10), \gamma=(3,5,7,9,10)$. Clearly $\alpha,\beta,\gamma \in I(d)$ and $\alpha \leq \beta \leq \gamma$. Again, as $\alpha, \beta, \gamma \in I(d)$ so by lemma  \ref{lemma.1}, $\widetilde{T}_\alpha, \widetilde{W}_\gamma$ both are star sets. We want to compute $Mult_{e_\beta}X_\alpha ^\gamma$. The following two diagrams show the negative and positive twisted chains $\widetilde{T}_\alpha =\{r_1,r_2\}$ and $\widetilde{W}_\gamma =\{s_1,s_2,s_3\}$ in $\bar{\beta} \times \beta$; and the set of $\lfloor r \rfloor$'s and $ \lceil r \rceil$'s for all $r \in \widetilde{T}_\alpha \cup \widetilde{W}_\gamma$. Note that, $s_1 = \lfloor s_1 \rfloor = \lceil s_1 \rceil$ and $s_3 = \lfloor s_3 \rfloor = \lceil s_3 \rceil$.
\end{eg}
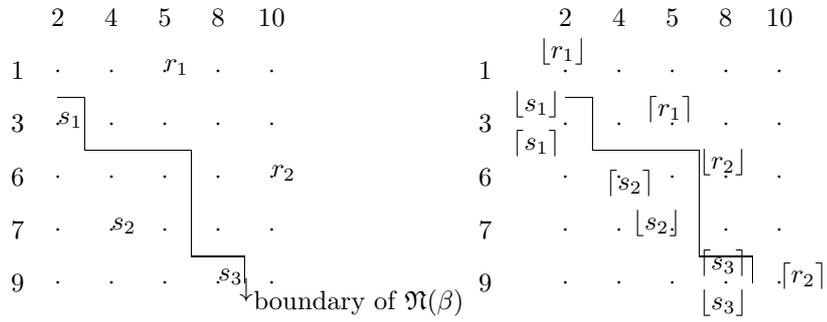
\begin{picture}(100,120)(0,0)
\matrixput(0,0)(20,0){5}(0,20){5}{\circle*{.5}}
\multiputlist(-15,0)(0,20){9,7,6,3,1}
\multiputlist(0,100)(20,0){2,4,5,8,10}
\put(40,80){$r_1$}
\put(80,40){$r_2$}
\put(0,60){$s_1$}
\put(20,20){$s_2$}
\put(60,){$s_3$}
\linethickness{.1pt}
\dottedline{0.1}(0,70)(10,70)(10,50)(50,50)(50,10)(70,10)(70,0)
\put(190,0){\matrixput(0,0)(20,0){5}(0,20){5}{\circle*{.5}}}
\put(175,0){\multiputlist(-15,0)(0,20){9,7,6,3,1}}
\put(190,0){\multiputlist(0,100)(20,0){2,4,5,8,10}}
\put(190,0){\linethickness{.1pt}
\dottedline{0.1}(0,70)(10,70)(10,50)(50,50)(50,10)(70,10)(70,0)
\put(-20,65){$\lfloor s_1 \rfloor$}
\put(-20,50){$\lceil s_1 \rceil$}
\put(50,-10){$\lfloor s_3 \rfloor$}
\put(50,5){$\lceil s_3 \rceil$}
\put(15,35){$\lceil s_2 \rceil$}
\put(25,20){$\lfloor s_2 \rfloor$}
\put(-10,85){$\lfloor r_1 \rfloor$}
\put(30,63){$\lceil r_1 \rceil$}
\put(80,0){$\lceil r_2 \rceil$}
\put(50,43){$\lfloor r_2 \rfloor$}}
\put(68,-5){$\downarrow$}
\put(70,-10){ boundary of $\mathfrak{N}(\beta)$}
\end{picture}

\vspace{1cm}
\noindent There are four non intersecting path families from $\lfloor r \rfloor$ to $ \lceil r \rceil$, $r \in \widetilde{T}_\alpha \cup \widetilde{W}_\gamma$, as shown below. Thus $Mult_{e_\beta}X_\alpha ^\gamma = 4$.

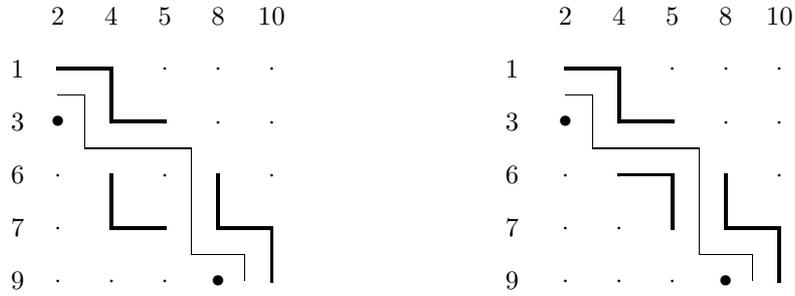
\begin{picture}(100,140)(0,0)
\matrixput(0,0)(20,0){5}(0,20){5}{\circle*{.5}}
\multiputlist(-15,0)(0,20){9,7,6,3,1}
\multiputlist(0,100)(20,0){2,4,5,8,10}
\linethickness{.1pt}
\dottedline{0.1}(0,70)(10,70)(10,50)(50,50)(50,10)(70,10)(70,0)
\put(0,60){\makebox(0,0){$\bullet$}}
\put(60,0){\makebox(0,0){$\bullet$}}
\put(190,0){\matrixput(0,0)(20,0){5}(0,20){5}{\circle*{.5}}}
\linethickness{1pt}
\dottedline{0.1}(20,40)(20,20)(40,20)
\linethickness{1pt}
\dottedline{0.1}(0,80)(20,80)(20,60)(40,60)
\linethickness{1pt}
\dottedline{0.1}(60,40)(60,20)(80,20)(80,0)
\put(185,0){\multiputlist(-15,0)(0,20){9,7,6,3,1}}
\put(190,0){\multiputlist(0,100)(20,0){2,4,5,8,10}}
\put(190,0){\linethickness{.1pt}
\dottedline{0.1}(0,70)(10,70)(10,50)(50,50)(50,10)(70,10)(70,0)}
\put(190,0){\put(0,60){\makebox(0,0){$\bullet$}}
\put(60,0){\makebox(0,0){$\bullet$}}}
\put(190,0){\linethickness{1pt}
\dottedline{0.1}(20,40)(40,40)(40,20)
\linethickness{1pt}
\dottedline{0.1}(0,80)(20,80)(20,60)(40,60)
\linethickness{1pt}
\dottedline{0.1}(60,40)(60,20)(80,20)(80,0)}
\end{picture}

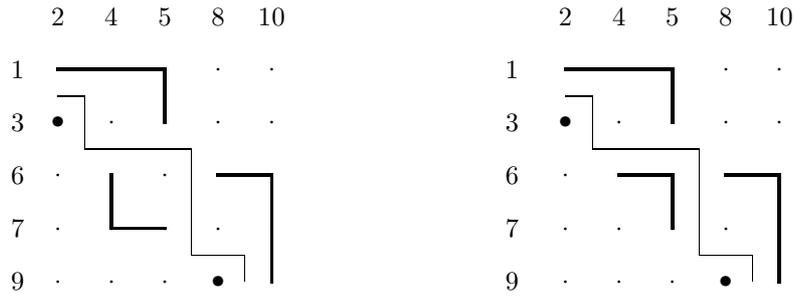
\begin{picture}(100,180)(0,0)
\matrixput(0,0)(20,0){5}(0,20){5}{\circle*{.5}}
\multiputlist(-15,0)(0,20){9,7,6,3,1}
\multiputlist(0,100)(20,0){2,4,5,8,10}
\linethickness{.1pt}
\dottedline{0.1}(0,70)(10,70)(10,50)(50,50)(50,10)(70,10)(70,0)
\linethickness{1pt}
\dottedline{0.1}(20,40)(20,20)(40,20)
\linethickness{1pt}
\dottedline{0.1}(0,80)(40,80)(40,60)
\linethickness{1pt}
\dottedline{0.1}(60,40)(80,40)(80,0)
\put(0,60){\makebox(0,0){$\bullet$}}
\put(60,0){\makebox(0,0){$\bullet$}}
\put(190,0){\matrixput(0,0)(20,0){5}(0,20){5}{\circle*{.5}}}
\put(185,0){\multiputlist(-15,0)(0,20){9,7,6,3,1}}
\put(190,0){\multiputlist(0,100)(20,0){2,4,5,8,10}}
\put(190,0){\linethickness{.1pt}
\dottedline{0.1}(0,70)(10,70)(10,50)(50,50)(50,10)(70,10)(70,0)}
\put(190,0){\put(0,60){\makebox(0,0){$\bullet$}}
\put(60,0){\makebox(0,0){$\bullet$}}}
\put(190,0){\linethickness{1pt}
\dottedline{0.1}(20,40)(40,40)(40,20)
\linethickness{1pt}
\dottedline{0.1}(0,80)(40,80)(40,60)
\linethickness{1pt}
\dottedline{0.1}(60,40)(80,40)(80,0)}
\end{picture}


\begin{thebibliography}{25}
\bibitem{kv}
 Kreiman, V.  : \textit{Local properties of Richardson varieties in the Grassmannian via a bounded Robinson-Schensted-Knuth correspondence}, J Algebra Comb (2008) 27: 351--382 .
 \bibitem{kr}
 Kodiyalam, V.,Raghavan, K.N.: \textit{Hilbert functions of points on Schubert varieties in Grassmannians}, J. Algebra 270(1), 28--54 (2003) .
\bibitem{gr}
Ghorpade, S., Raghavan, K.N. : \textit{Hilbert functions of points on Schubert varieties in the symplectic Grassmannians}, Trans. Am. Math. Soc. \textbf{358}(12), 5401--5423, 2006 .

\bibitem{ru}
Ray, P., Upadhyay, S. : Initial ideals of tangent cones to Richardson varieties in the Symplectic Grassmannian, arXiv:1905.01660.
\end{thebibliography}
\end{document}